\documentclass[11pt]{amsart}
\usepackage{amsmath,amssymb,amsthm}
\usepackage{hyperref}
\hypersetup{hidelinks=true}

\usepackage{bbm}

\usepackage[utf8]{inputenc}

\usepackage{graphicx}
\usepackage{enumerate}
\usepackage{amsfonts,amssymb,latexsym}
\usepackage{accents}
\usepackage{afterpage}
\usepackage[style=alphabetic,maxbibnames=99, backend=bibtex]{biblatex}

\newcommand{\ep}{\varepsilon}
\newcommand{\N}{\mathbb{N}}

\newcommand{\R}{\mathbb{R}}

\newcommand\restr[2]{{
  \left.\kern-\nulldelimiterspace 
  #1 
  \vphantom{\big|} 
  \right|_{#2} 
  }}

\DeclareMathOperator{\co}{co}

\newtheorem{theorem}{Theorem}[section]
\newtheorem{lemma}[theorem]{Lemma}

\newtheorem{prop}[theorem]{Proposition}
\newtheorem{corollary}[theorem]{Corollary}
\theoremstyle{definition}
\newtheorem{definition}[theorem]{Definition}

\theoremstyle{remark}
\newtheorem{remark}[theorem]{Remark}
\numberwithin{equation}{section}

\def\fnote#1{\footnote}

\def\R{{\mathbb R}}

\def\ignora#1{}
\def\n3#1{\left\vert  \! \left\vert \! \left\vert \, #1 \, \right\vert \!
  \right\vert \! \right\vert }

\include{draft}

\begin{document}



\title[Compact retracts]{Retractions and the bounded approximation property in Banach spaces}

\author{Petr H\'ajek}\thanks{This research was supported by CAAS CZ.02.1.01/0.0/0.0/16-019/0000778
 and by the project  SGS21/056/OHK3/1T/13.}
\address[P. H\'ajek]{Czech Technical University in Prague, Faculty of Electrical Engineering.
Department of Mathematics, Technická 2, 166 27 Praha 6 (Czech Republic)}
\email{ hajek@math.cas.cz}

\author{ Rub\'en Medina}\thanks{The second author research has also been supported by MICINN (Spain) Project PGC2018-093794-B-I00 and MIU (Spain) FPU19/04085 Grant.}
\address[R. Medina]{Universidad de Granada, Facultad de Ciencias.
Departamento de An\'{a}lisis Matem\'{a}tico, 18071-Granada
(Spain); and Czech technical University in Prague, Faculty of Electrical Engineering.
Department of Mathematics, Technická 2, 166 27 Praha 6 (Czech Republic)}
\email{rubenmedina@ugr.es}
\urladdr{\url{https://www.ugr.es/personal/ae3750ed9865e58ab7ad9e11e37f72f4}}

\maketitle 

\begin{abstract}
In the present paper we prove that  a necessary condition for a Banach space $X$ to admit a generating compact Lipschitz retract $K$, which satisfies an additional mild assumption on its shape, is that $X$ enjoys the Bounded Approximation Property. This is a partial solution to a question raised by Godefroy and Ozawa.
\end{abstract}

\section{Introduction}

 In this note, we focus on the problem of the existence of Lipschitz retractions onto  compact and generating subsets of a Banach space $X$. A subset $C$ of a Banach space $X$ is called generating if the closed linear span of $C$ is the whole  space $X$.

In \cite{GO14}  (and subsequently in  \cite{God15}, \cite{God215}, \cite{God20},  and \cite{GP19})
Godefroy and Ozawa asked whether every separable Banach space $X$ admits a generating convex and compact Lipschitz retract (GCCR for short) $K$. The main result of our paper (Theorem \ref{theosmall}) hints towards a negative answer, since we prove that the existence of a generating compact Lipschitz retract  $K$ (not necessarily convex) satisfying some additional (not exceedingly restrictive)  assumptions  implies that $X$ has the Bounded Approximation Property. 
If we moreover assume that $K$ is convex, the additional assumptions take on a simpler form
(Corollary \ref{cor}).

Of course, thanks to Enflo we know that BAP is not true for every separable Banach space \cite{Enf73}.

Our argument is a variant of that in \cite{HM21}, where a much stronger flatness assumption  on $K$ was used to imply the $\pi$-property of $X$. Roughly speaking, the flatness assumption is a quantitative 
condition on the sequence of distances  of the given compact set $K$ to its finite-dimensional sections. The compactness condition itself implies that this sequence tends to zero. In our result, we assume that it  is controlled by a geometric sequence.
This hints to the possibility of the existence of a  counterexample to the question of Godefroy and Ozawa.

It should be noted that in \cite{HM21} we construct a GCCR  in every Banach space with an FDD, but their
existence in every Banach space with the weaker property BAP is not clear to us. 

The proof of the main Theorem \ref{theosmall} is based on several deep ingredients and follows the same scheme developed in Section 4 of \cite{HM21}. We produce the finite rank operators approximating  the identity by means of the Lipschitz retraction in a non-constructive way, only proving their existence. This is done in several steps. First, we deal with coarse Lipschitz maps given by the composition of the retraction with compressions onto the n-dimensional sections of the space (given by the nearest point map, not necessarily unique). Since this maps may not be Lipschitz, we need to ''Lipschitzize'' them making use of a technique by Bourgain \cite{Bou87} (streamlined later by Begun \cite{B99}). This technique works only in the finite dimensional setting so the crucial part here is to prove that we can restrict each map to some finite dimensional subspace whose dimension is under control. To do so we present a generalization of Lemma 4.2 in \cite{HM21} originally due to Vitali Milman. Finally, once we have the Lipschitz maps, it is a matter of using the differentiation theory and averaging in the finte dimensional setting to obtain the uniformly bounded finite rank operators pointwise converging to the identity.

Let us now review the basic definitions that will be used in the next section.

\begin{definition}
Let $X$ be a Banach space. If there is a uniformly bounded net of finite rank operators $(T_\alpha)$ on $X$ tending strongly to the identity on $X$, then we say that $X$ has the bounded approximation property (BAP for short). If $\lambda\ge1$ is a uniform bound for the net then we say that $X$ has the $\lambda$-bounded approximation property ($\lambda$-BAP for short).

In the separable case we may assume that the net is actually a sequence.
\end{definition}

\begin{definition}
A map $T$ from a metric space $M$ into another metric space $N$ is said to be Lipschitz if there exists some $\lambda>0$ such that
$$d(T(x),T(y))\le \lambda d(x,y) \;\;\;\;\forall x,y\in M.$$
We say that $\lambda$ is the Lipschitz constant for $T$ and we call the infimum of all Lipschitz constants for $T$ the Lipschitz norm of $T$, that is,
$$||T||_{Lip}=\inf\big\{\lambda>0,\text{ Lipschitz constant for }T\big\}=\sup\limits_{x,y\in M, x\ne y}\frac{d\big(T(x),T(y)\big)}{d(x,y)}.$$
If $\lambda>0$ is a Lipschitz constant for $T$ then we say that $T$ is $\lambda$-Lipschitz.
\end{definition}

A retraction from a metric space $M$ onto its subset $N$ is a map $R:M\to N$ whose restriction to $N$ is the identity on $N$.

For the background in nonlinear analysis we refer the reader to the authoritative monograph \cite{BL2000}. For the remaining general concepts and results of Banach space theory we refer to \cite{Fab1} and the background on the approximation properties can be found in \cite{Cas01}.

\section{Main result}

In this section we introduce the key concept of flatness for a compact set  and prove the main result of this paper.
 Namely, every Banach space with a Lipschitz retraction onto a flat and generating set with nonempty local interior has the Bounded Approximation Property (BAP for short). To do so, we follow the ideas in \cite{HM21} Section 4. In particular we need to generalize Lemma 4.2 of \cite{HM21}, a crucial result about projection constants due to Vitali Milman, to the setting of BAP type operators.

For that purpose, we will need the next well known result.
\begin{lemma}\label{MilSecht}
Let $n\in \N$ and $\ep>0$. If $(F, ||\cdot||)$ is a Banach space of dimension $n$ then there exists a renorming $|\cdot|$ of $F$ such that $(F,|\cdot|)$ embeds isometrically in $\ell_\infty^N$ where $N=(3+\ep)^n$, and
$$|x|\le||x||\le|x|\frac{(2+\ep)}{\ep}.$$
\begin{proof}
By Lemma 2.6 of \cite{MS86}, we know that there exists an $\frac{2}{2+\ep}$-net $\{x_1^*,\dots,x_N^*\}$ in $S_{X^*}$. Just consider the norm $|x|=\max\limits_{i\in\{1,\dots,N\}}x_i^*(x)$.
\end{proof}
\end{lemma}

For a triplet $X,F,E$ of Banach spaces such that $X\supset F\supset E$ and both $F,E$ are finite dimensional, we define the BAP constant of $(X,F,E)$ as
$$\lambda(X,F,E)=\inf\{||T||\;:\;T\in\mathcal{L}(X,F),\;\restr{T}{E}=Id\}.$$
The next Theorem is a  generalization of Milman's Lemma 4.2  in \cite{HM21}.

\begin{theorem}\label{milmangen2}
Let $X$ be a separable Banach space, and $E,F$ be finite dimensional subspaces of $X$ such that $E\subset F$. Then, for every $\ep>0$ and $\rho>1$ there is  $G\subset X$, a subspace containing $E$, such that $dim(G)\le (3+\ep)^{dim(F)}$ and
$$\lambda(X,F,E)\le \rho\frac{2+\ep}{\ep}\lambda(G,F,E).$$
\end{theorem}

To prove Theorem \ref{milmangen2}, we rely on a variant of the trace duality introduced in \cite{JKM79}.
Given $T\in\mathcal{L}(E,F)$ we define the norm 
$$||T||_{X}=\inf\{||\widetilde{T}||\;:\;\widetilde{T}\in\mathcal{L}(X,F),\;\restr{\widetilde{T}}{E}=T\}.$$
Similarly, if  $S\in \mathcal{L}(F,E)$ we define the norm
$$||S||_{\Lambda X}=||i_{E,X}S||_{\Lambda}$$
where $||\cdot||_{\Lambda}$ refers to the nuclear norm in $\mathcal{N}(F,X)$ and $i_{Y,Z}:Y\to Z$ is the inclusion whenever $Y\subset Z$. We have the following duality.

\begin{prop}\label{propduality}
Let $X$ be a Banach space and $E\subset F$ finite dimensional subspaces of $X$. Then,
$$(\mathcal{L}(E,F),||\cdot||_{X})^*\equiv(\mathcal{L}(F,E),||\cdot||_{\Lambda X}),$$
where the duality is given by the trace, that is for $T\in\mathcal{L}(E,F)$
$$||T||_X=\sup\limits_{\substack{S\in\mathcal{L}(F,E)\\||S||_{\Lambda X}\le1}}tr(ST).$$
\begin{proof}
If $S\in \mathcal{L}(F,E)$ then by the trace duality of $\mathcal{L}(X,F)$ and $\mathcal{N}(F,X)$ we know that
$$||S||_{\Lambda X}=||i_{E,X}S||_{\Lambda}=\sup\limits_{\substack{T\in\mathcal{L}(X,F)\\||T||\le1}}tr(Ti_{E,X}S)=\sup\limits_{\substack{T\in\mathcal{L}(E,F)\\||T||_X\le1}}tr(TS).$$
Our claim follows now directly.
\end{proof}
\end{prop}

\begin{proof}[Proof of Theorem \ref{milmangen2}]
By Proposition \ref{propduality}, for every superspace $Y\supset F$,
\begin{equation}\label{maineq}\begin{aligned}\lambda(Y,F,E)=&||i_{E,F}||_Y=\sup\limits_{\substack{T\in\mathcal{L}(F,E)\\||T||_{\Lambda Y}=1}}tr(\restr{T}{E})=\sup\limits_{\substack{T\in\mathcal{L}(F,E)}}tr\bigg( \frac{\restr{T}{E}}{||T||_{\Lambda Y}} \bigg)\\=&\sup\limits_{\substack{T\in\mathcal{L}(F,E)}}\frac{1}{\bigg|\bigg| \frac{T}{tr(\restr{T}{E})} \bigg|\bigg|_{\Lambda Y}}=\frac{1}{\inf\limits_{\substack{T\in \mathcal{L}(F,E)\\tr(\restr{T}{E})=1}}||T||_{\Lambda Y}}.\end{aligned}\end{equation}

 Let us first take $\mu\in(1/\rho,1)$ and set $n=dim(F)$. Now we choose
$$\delta=\frac{\ep}{\lambda(X,F,E)(2+\ep)}(\mu-1/\rho)>0,$$
and an operator  $S\in\mathcal{L}(F,E)$ such that
$$\inf\limits_{\substack{T\in\mathcal{L}(F,E)\\tr(\restr{T}{E})=1}}||T||_{\Lambda X}\ge ||S||_{\Lambda X}-\delta\;\;\text{ and }\;\;tr(\restr{S}{E})=1.$$
We also take the norm $|\cdot|$ given by Lemma \ref{MilSecht} so that $(F,|\cdot|)$ is isometrically a subspace of $\ell_\infty^{\varphi(n)}$ where $\varphi(n)=\Big[\big( 3+\ep \big)^n\Big]$.  Then, for every superspace $Y\supset F$,
$$ ||S||_{\Lambda Y}\le|S|_{\Lambda Y}\le \frac{2+\ep}{\ep}||S||_{\Lambda Y}, $$
where $|\cdot|_{\Lambda Y}$ denotes the norm $||\cdot||_{\Lambda Y}$ but taking $(F,|\cdot|)$ as the domain of the operators instead of $(F,||\cdot||)$.
It is well-known (Proposition 47.6 in \cite{Tre06}) that $i_{E,X}S$ admits an extension $\widetilde{S}:\ell_\infty^{\varphi(n)}\rightarrow X$ almost preserving the nuclear norm, that is $|S|_{\Lambda X}\ge\mu||\widetilde{S}||_\Lambda$. By Proposition 8.7 in \cite{Tom89} we know that there exist $x_1,\dots,x_{\varphi(n)}\in X$ such that $\widetilde{S}=\sum\limits_{i=1}^{\varphi(n)}e_i^*\otimes x_i$ and
$$||\widetilde{S}||_{\Lambda}=\sum\limits_{i=1}^{\varphi(n)}||x_i||,$$
where $e_i^*\in\big(\ell_\infty^{\varphi(n)}\big)^*$ are the coordinate functionals.
Just considering $G=[x_i]_{i=1}^{\varphi(n)}$, we may see that
$$\begin{aligned}||S||_{\Lambda X}&\ge \frac{\ep}{2+\ep}|S|_{\Lambda X}\ge\frac{\ep\mu}{2+\ep}||\widetilde{S}||_{\Lambda}\ge\frac{\ep\mu}{2+\ep}|S|_{\Lambda G}\ge\frac{\ep\mu}{2+\ep}||S||_{\Lambda G}.\end{aligned}$$
Finally, taking into account that $\lambda(G,F,E)\le\lambda(X,F,E)$ and using \eqref{maineq} with $Y=X$ and $Y=G$,
$$\begin{aligned}\lambda(X,F,E)&\le\frac{1}{||S||_{\Lambda X}-\delta}\le \frac{1}{\frac{\ep\mu}{2+\ep}||S||_{\Lambda G}-\delta}\le\frac{1}{\frac{\ep\mu}{2+\ep}\inf\limits_{tr(\restr{T}{E})=1}||T||_{\Lambda G}-\delta}\\&=\lambda(G,F,E)\frac{1}{\frac{\ep\mu}{2+\ep}-\delta\lambda(G,F,E)}\le\rho\frac{2+\ep}{\ep}\lambda(G,F,E).\end{aligned}$$
\end{proof}

\begin{definition}
Let $X$ be a separable Banach space, $K\subset X$ be a bounded generating subset and 
$\beta=(e_n)_{n\in\N}\subset span(K)$ a fundamental sequence of $X$. We define the sequence of heights $(h_n^\beta)_{n\in\N}$ of $K$ relative to $\beta$ as
$$h_n^\beta=\sup\big\{d(x,[e_i]_{i=1}^n)\;:\;x\in K\big\}\;\;\;\;\forall n\in\N.$$
If there exists a fundamental sequence $\beta$ and $\ep>0$ satisfying that 
$$0=\liminf_n h_n^\beta(3+\ep)^n,$$
then we say that $K$ is $\beta$-flat.  If $0=\lim h_n^\beta(3+\ep)^n$ then we say that $K$ is basically $\beta$-flat. We will say that $K$ is flat (resp. basically flat) whenever it is $\beta$-flat (resp. basically $\beta$-flat) for some fundamental sequence of $X$ $\beta\subset span(K)$.
\end{definition}

Let us point out that $K$ is flat if and only if there is (possibly a different) $\ep>0$ and $\sigma:\N\to\N$ strictly increasing such that $h^\beta_{\sigma(n)}\le(3+\ep)^{-\sigma(n)}$. In the case when $K$ is basically flat we may just take $\sigma(n)=n$ for every $n\in\N$.

Let $K$ be a compact set in an infinite dimensional Banach space $X$. For every fundamental sequence $\beta$ in $X$ the sequence of heights relative to $\beta$ necessarily converges to zero. The above definition of flat set requires one of these sequences of heights to decrease faster than a geometric sequence of ratio smaller than $1/3$.

\begin{definition}
Let $X$ be a separable Banach space, a subset $C\subset X$ and a fundamental sequence $\beta=(e_n)$ of $X$. We say that $C$ has nonempty local $\beta$-interior whenever
$$int(C\cap[e_i]_{i=1}^n)\neq\emptyset\;\;\;\;\forall n\in\N, $$
where the interior is taken in the topology relative to $[e_i]_{i=1}^n$.
\end{definition}

\begin{theorem}\label{theosmall}
Let $X$ be a Banach space, $K$ a generating subset of $X$ and $\beta=(e_n)\subset span(K)$ a fundamental sequence of $X$ such that $K$ is $\beta$-flat and has nonempty local $\beta$-interior. Then a necessary condition for the existence of a Lipschitz retraction from $X$ onto $K$
is that $X$ has the BAP.
\begin{proof}
Assume that there is a Lipschitz retraction from $X$ onto a generating subset $K\subset X$ which is  basically $\beta$-flat and has a nonempty local $\beta$-interior. Take $\ep>0$ such that $h_n(3+\ep)^n\to0$ where the $h_n$ are the heights of $K$ relative to $\beta$, and let $\phi(n)=\Big[\big(3+\ep\big)^{n}\Big]$ and $E_n=[e_i]_{i=1}^n$. We define the inner radius relative to $\beta$ as
$$r_n=\sup\big\{r\ge0\;:\;B_{E_n}(x,r)\subset K\cap E_n\;,\;x\in X\big\}\;\;\;\;\forall n\in\N.$$
As $K$ has nonempty $\beta$-interior $r_n>0$ for every $n\in\N$ so there must exist a strictly increasing $\gamma:\N\to\N$ such that
$$h_{\gamma(n)}(2+\phi(\gamma(n))\le \frac{r_{n}}{n^3}\;\;\;\;\;\forall n\in\N.$$
Theorem \ref{milmangen2} guarantees that for every $n\in\N$ there is a finite dimensional subspace $G_n\subset X$ of dimension $dim(G_n)=\varphi(n):=\phi(\gamma(n))$ such that for every linear operator $T: G_n\rightarrow E_{\gamma(n)}:=F_n$ with $\restr{T}{E_n}=Id_{E_n}$, the inequality $||T||\ge \lambda(X,F_n,E_n)\frac{\ep}{2(2+\ep)}$ holds. Assume that $R:X\rightarrow K$ is the Lipschitz retraction. Then taking the nearest point map $C_n:K\rightarrow F_n$  (it may not be unique), we define $\widetilde{R}_n=\restr{(C_n\circ R)}{G_{n}}:G_{n}\rightarrow F_n$ for every $n\in\N$. Now,
$$||\widetilde{R}_n(x)-\widetilde{R}_n(y)||\le ||R||\bigg( ||x-y||+\frac{2h_{\gamma(n)}}{||R||} \bigg)\;\;\;\;\forall x,y\in G_{n}.$$
So by the Proposition of \cite{B99}, for every $\tau>0$, there is a Lipschitz mapping 
$$R_{n,\tau}:G_{n}\rightarrow F_n,$$
such that
$$||R_{n,\tau}||_{Lip}\le||R||\bigg( 1+\frac{\varphi(n)h_{\gamma(n)}}{\tau||R||} \bigg),$$
$$||R_{n,\tau}(x)-\widetilde{R}_n(x)||\le||R||\bigg( \tau+\frac{2h_{\gamma(n)}}{||R||} \bigg)\;\;\;\;\forall x\in G_{n}.$$

For the rest of the argument we fix $n\in\N$ and  $x_n\in K_n:=K\cap E_n$ such that $B_{E_n}(x_n,r_n)\subset K_n$ (it exists by the definition of the inner radius and the compactness).
Now we choose $\tau_n=\frac{\varphi(n)h_{\gamma(n)}}{||R||}$, and define $R_n:G_n\rightarrow F_n$ by $R_n(x)=R_{n,\tau_n}(x+x_n)-x_n$. 
If $x+x_n\in K_n$ then $\widetilde{R}_n(x+x_n)=x+x_n$ acts as an identity.
Hence, we have that for every $x\in K_n+\{-x_n\}$
$$||R_n(x)-x||=||R_{n,\tau_n}(x+x_n)-\widetilde{R}_n(x+x_n)||\le h_{\gamma(n)}(\varphi(n)+2)=:\rho_n.$$
Now, let $(a_i,a_i^*)_{i=1}^{\varphi(n)}$ be a normalized linear basis for $G_n$ with projections $(S_i)_{i=1}^{\varphi(n)}$ such that $(a_i,a_i^*)_{i=1}^{n}$ is an Auerbach basis for $E_n$. Then, 
$$B_n=r_n\co\big(\{\pm a_i\;,\;i=1,\dots,n\}\big)\subset K_n+\{-x_n\}.$$
Fix a sequence $(\delta_k)$ of positive numbers converging to zero. Now we define for every $k\in\N$ the compact
$$B_{n,k}=B_n+\delta_k\sum\limits_{i=n+1}^{\varphi(n)}[-a_i,a_i]\subset G_n.$$
The geometrical shape of this set can be described as a $r_n$-multiple of the unit ball of $\ell_1^{n}$ located in $E_n$, a sort of a base of a hypercylinder, times a hypercube of side length $2\delta_k$ protruding into the remaining dimensions of $G_n$.
Letting $\delta_k$ go to zero of course means that this set gets squashed down to its base in $E_n$.
For any $i\in\{1,\dots,n\}$ we denote $(B_{n,k})_i=(Id-a_i^*\otimes a_i)(B_{n,k})$. This set is just a one-codimensional section (or a projection of rank $\varphi(n)-1$) of $B_{n,k}$ which reduces the base by one coordinate.

In order to recover the shape $B_{n,k}$ from its section $(B_{n,k})_i$ we pass from any point  $x^i\in (B_{n,k})_i$ to the boundary point of $B_{n,k}$ which got projected onto it. The newly acquired coordinate vector will then be denoted by $x_i(x^i)$ and 
given by the formula

$$x_i(x^i)=\bigg( r_n-\sum\limits_{\substack{j=1\\j\neq i}}^{n}a_j^*(x^i) \bigg)a_i,\;\; x^i\in (B_{n,k})_i.$$

For convenience in our computations, we also introduce the quantity
$$z_i(x^i)=R_n(x^i+x_i(x^i))-R_n(x^i-x_i(x^i))-2x_i(x^i).$$

As $S_{n}$ is a linear projection, we know that
$-S_{n}\big(x^i+x_i(x^i)\big)+S_{n}\big(x^i-x_i(x^i)\big)=-2x_i(x^i)$. Using this and the triangle inequality with four terms 
we have that
$$\begin{aligned}||z_i(x^i)||\le&\big|\big|R_n\big(x^i+x_i(x^i)\big)-R_n\big(S_{n}\big(x^i+x_i(x^i)\big)\big)\big|\big|\\&+\big|\big|R_n\big(S_{n}\big(x^i+x_i(x^i)\big)\big)-S_{n}\big(x^i+x_i(x^i)\big)\big|\big|\\&+\big|\big|R_n\big(S_{n}\big(x^i-x_i(x^i)\big)\big)-R_n\big(x^i-x_i(x^i)\big)\big|\big|\\&+\big|\big|S_{n}\big(x^i-x_i(x^i)\big)-R_n\big(S_{n}\big(x^i-x_i(x^i)\big)\big)\big|\big|\\\le&2||R_n||\varphi(n)\delta_k+2\rho_n.\end{aligned}$$
For a set $J\subset \{1,\dots,\varphi(n)\}$ with $\# J=m\ge1$, we define the measure in $[a_i]_{i\in J}$ as 
$$\lambda^m_J(A)=\lambda_m\bigg(\Big(\big( a_i^*\big)_{i\in J}\Big)\big(A\big)\bigg)\;\;\;\;\forall A\in \mathcal{M}^m_J, $$
where $\lambda_m$ is the Lebesgue measure in $\R^m$ and $\mathcal{M}^m_J=\bigg\{A\subset [a_i]_{i\in J}\;:\;\Big(\big( a_i^*\big)_{i\in J}\Big)\big(A\big) \text{ is a Lebesgue measurable subset of }\R^m\bigg\}$. If $J=\{1,\dots,\varphi(n)\}$ we denote $\lambda^{\varphi(n)}=\lambda^{\varphi(n)}_J$. For every $i\in\{1,\dots,\varphi(n)\}$, if $J=\{1,\dots,\varphi(n)\}\setminus\{i\}$ then we denote $\lambda^{\varphi(n)-1}_i=\lambda^{\varphi(n)-1}_{J}$.
We are ready to define  the linear operators $T_{n,k}:G_n\rightarrow E_n$, for every $k\in\N$, as
$$T_{n,k}(v)=\frac{1}{\lambda^{\varphi(n)}(B_{n,k})}\int_{B_{n,k}}dR_n(x)[v]d\lambda^{\varphi(n)}(x).$$

In \cite{Bra+14} pg. 47
the volumes of $\ell_p^n$ balls $B_p^n$ have been computed as $|B_p^n|=\frac{2^n\Gamma(\frac1p+1)^n}{\Gamma(\frac np+1)}$. Using this result for $p=1$ (for the base part of our set $B_{n,k}$) and  standard
properties of the Lebesgue measure we obtain the following values for our sets for arbitrary $i\in\{1,\dots,n\}$
$$\lambda^{\varphi(n)-1}_i\big((B_{n,k})_i\big)=\frac{2^{\varphi(n)-1}r_n^{n-1}\delta_k^{\varphi(n)-n}}{(n-1)!},$$
$$\lambda^{\varphi(n)}\big(B_{n,k}\big)=\frac{2^{\varphi(n)}r_n^{n}\delta_k^{\varphi(n)-n}}{n!},$$
so the quotient is equal to
$$\frac{\lambda^{\varphi(n)-1}_i((B_{n,k})_i)}{\lambda^{\varphi(n)}(B_{n,k})}=\frac{n}{2r_n}.$$
Note that the expression
$$R_n(x^i+x_i(x^i))-R_n(x^i-x_i(x^i))=z_i(x^i)+2x_i(x^i)$$
represents the difference of the values of the operator $R_n$ between the endpoints of a segment cutting through $B_{n,k}$, which passes through the point $x^i$ with direction $a_i$.
As $B_{n,k}=\{x^i+x_i\in G_n\;:\;x^i\in(B_{n,k})_i\;,\;x_i\in[-x_i(x^i),x_i(x^i)]\}$, thanks to Fubini's Theorem and the Fundamental Theorem of Calculus applied to the $i$-th coordinate, we can compute for each $i\in\{1,\dots,n\}$
$$\begin{aligned}\big|\big|a_i-T_{n,k}(a_i)\big|\big| 
&=\bigg|\bigg|a_i- \frac{1}{\lambda^{\varphi(n)}(B_{n,k})}\int_{B_{n,k}}dR_n(x)[v]d\lambda^{\varphi(n)}(x)\bigg|\bigg|\\ 
&=\bigg|\bigg|a_i- \frac{1}{\lambda^{\varphi(n)}(B_{n,k})}\int_{(B_{n,k})_i}z_i(x^i)+2x_i(x^i)d\lambda^{\varphi(n)-1}_i(x^i) \bigg|\bigg|\\&=\bigg|\bigg|\frac{1}{\lambda^{\varphi(n)}(B_{n,k})} \int_{(B_{n,k})_i}z_i(x^i)d\lambda^{\varphi(n)-1}_i(x^i) \bigg|\bigg|\\
&\le\frac{\lambda^{\varphi(n)-1}_i((B_{n,k})_i)}{\lambda^{\varphi(n)}(B_{n,k})}\big(2||R_n||\varphi(n)\delta_k+2\rho_n\big)\\
&=\frac{n||R_n||\varphi(n)\delta_k+n\rho_n}{r_n}.
\end{aligned}$$
We may assume that $T_{n,k}$ pointwise converge in $k$ and define $T_n(x)=\lim\limits_{k\to\infty}T_{n,k}(x)$ for every $x\in G_n$. It is a linear operator from $G_n$ to $F_n$ satisfying 
$$||T_n(a_i)-a_i||\le\frac{nh_{\gamma(n)}(\varphi(n)+2)}{r_n}\;\;\;\;\forall i\in\{1,\dots,n\}.$$
As $||a_i^*||=1$, $i\in\{1,\dots,n\}$, we obtain that for every $x\in E_n$,
$$||T_n(x)-x||=\bigg|\bigg| \sum\limits_{i=1}^{n}a_i^*(x)(T_n(a_i)-a_i) \bigg|\bigg|\le\frac{h_{\gamma(n)}(\varphi(n)+2){n}^2}{r_n}||x||\le\frac{||x||}{n}.$$
Finally, we take a projection $P_n:F_n\to E_n$ of norm $n$ and construct as in \cite{Cas01} Lemma 3.2 the linear mapping
$$\widetilde{T}_n=\Big(\big(\restr{T_n}{E_n}\big)^{-1}\circ P_n+Id_{F_n}-P_n\Big)\circ T_n:G_n\rightarrow F_n,$$
satisfying
$$||\widetilde{T}_n-T_n||\le\Big|\Big| \Big(\restr{T_n}{E_n}\Big)^{-1}-Id_{E_n} \Big|\Big|\cdot||P_n||\cdot||T_n||\le\frac{n}{n-1}||R_n||_{Lip} \le\frac{2n}{n-1}||R||_{Lip}.$$
This implies that $X$ has the BAP since
$$\lambda(X,F_n,E_n)\le\frac{2(2+\ep)}{\ep}||\widetilde{T}_n||\le\frac{2(2+\ep)(4n-2)}{\ep(n-1)}||R||_{\text{Lip}}\;\;\;\;\forall n\in\N.$$
The proof in the case of basically flat sets is completed. 
The general case of a flat set is essentially the same, considering just the sequence $\sigma(n)$ in place  of that of all integers $n$.
\end{proof}
\end{theorem}

\begin{remark}
We didn't strive for the best possible bound on $\lambda(X,F_n,E_n)$. In fact, using the same argument but with sharper estimates it is possible to prove that $X$ has the $\lambda^+$-BAP (meaning that $X$ has the $(\lambda+\delta)$-BAP for every $\delta>0$) where
$$\lambda=\frac{2+\ep}{\ep}||R||_{\text{Lip}}.$$
However, this is the best possible bound we are able to find using this approach.
\end{remark}

\begin{corollary}\label{cor}
Let $X$ be a separable Banach space. If there exists a Lipschitz retraction from $X$ onto a generating convex flat subset $K$ with $0\in K$ then $X$ has the BAP.
\begin{proof}
Let $\beta=(e_n)\subset span(K)$ be a fundamental sequence of $X$ relative to which $K$ is $\beta$-flat. During this proof we set $E_n=[e_i]_{i=1}^n$ for every $n\in\N$. Our aim here is to prove that there is a Lipschitz retraction onto a subset satisfying the assumptions of Theorem \ref{theosmall}. In particular, we will see that there is $y\in K$ such that $K-y$ is $\beta$-flat and has nonempty local $\beta$-interior.

Let us first show that if $K$ is $\beta$-flat then $K-z$ is $\beta$-flat whenever $z\in K$. To see this, let $h_n$ be the sequence of heights for $K$ and $\widetilde h_n$ the sequence of heights for $K-z$. Given $x\in K$ we take $y_x,y_z\in E_n$ such that $d(x,E_n)=||x-y_x||$ and $d(z,E_n)=||z-y_z||$. Clearly, $z+(y_x-y_z)\in z+E_n$. Hence,
$$d(x-z,E_n)=d(x,z+E_n)\le ||x-(z+y_x-y_z)||\le||x-y_x||+||z-y_z||\le 2h_n.$$
This implies that $\widetilde h_n\le2h_n$. Let us take $\ep>0$ such that $\lim\inf h_n(3+\ep)^n=0$. Then,
$$0\le\lim\inf \widetilde h_n(3+\ep)^n\le2\lim\inf h_n(3+\ep)^n=0,$$
which proves that $K-z$ is $\beta$-flat.

We pass on to find $y\in K$ such that $K-y$ has nonempty local $\beta$-interior. For every $n\in\N$ there are elements $x_1,\dots,x_{k_n}\in K$ such that $E_n\subset span(x_1,\dots,x_{k_n})$. We consider then
$$y_n=\frac{1}{k_n+1}\sum\limits_{i=1}^{k_n}x_i,$$
which is again in $K$ (since $0\in K$). Now, we claim that
\begin{equation}\label{int1}
int\big((K-y_n)\cap E_n\big)\neq\emptyset\;\;\;\forall n\in\N.
\end{equation}
To show \eqref{int1} it is enough to prove that for every $z\in span(x_1,\dots, x_{k_n})$ there is $\ep>0$ such that $y_n+\ep z\in K$. Let us take $z=\sum\limits_{i=1}^{k_n}\lambda_ix_i\in span(x_1,\dots, x_{k_n})$ and $\ep>0$ satisfying that
$$\frac{1}{k_n+1}+\ep\min\limits_{i=1,\dots,k_n}(\lambda_i)>0\;\;\;\;\text{and}\;\;\;\;\frac{k_n}{k_n+1}+\ep\sum\limits_{i=1}^{k_n}\lambda_i\le1.$$
The claim follows immediately since
$$y_n+\ep z=\sum\limits_{i=1}^{k_n}\Big(\frac{1}{k_n+1}+\ep\lambda_i\Big)x_i\in K.$$
The statement \eqref{int1} is now proven. Let us consider the element
$$y=\sum\limits_{k\in\N}2^{-k}y_k.$$
We finish the proof by showing that $(K-y)\cap E_n$ has nonempty interior for every $n\in\N$. Again, it is enough to see that for every $z\in E_n$ there is $\ep>0$ such that $y+\ep z\in K$. Given $n\in\N$ we take $z\in E_n$ and set $z_n=\sum\limits_{k\neq n}\frac{2^{-k}}{1-2^{-n}}y_k\in K$ so that $y=2^{-n}y_n+(1-2^{-n})z_n$. By \eqref{int1} there is $\ep>0$ such that $y_n+\ep z\in K$. Hence, letting $\ep'=2^{-n}\ep$ we are done since
$$y+\ep'z=2^{-n}(y_n+\ep z)+(1-2^{-n})z_n\in K.$$
\end{proof}
\end{corollary}

\bigskip

\printbibliography
\end{document}